\documentclass[11pt]{article}
\usepackage{amsmath,amssymb,amsthm,mathtools,hyperref,geometry}
\newtheorem{theorem}{Theorem}
\newtheorem{lemma}[theorem]{Lemma}
\newtheorem{proposition}[theorem]{Proposition}
\theoremstyle{remark}
\newtheorem{remark}[theorem]{Remark}

\title{Proof of Conjecture 19 of Ballantine, Beck, Merca, and Sagan on Elementary Symmetric Partitions}
\author{Arnav Garg\\
\small Birla Institute of Technology and Science Pilani}
\date{}

\begin{document}
\maketitle

\begin{abstract}
Ballantine, Beck, Merca, and Sagan \cite{BBMS} conjectured four identities, collectively Conjecture~19, relating the image of the map $\mathrm{pre}_k$ on integer partitions to four OEIS sequences. We prove parts~(i) and~(iii) unconditionally, prove part~(iv) unconditionally using the injectivity of $\mathrm{pre}_2$ on partitions of~$n$ (which was Conjecture~1 of the same paper and was proved by Li~\cite{Li}), and show that this injectivity is in fact equivalent to part~(iv). For part~(ii) we prove the partition-theoretic half unconditionally and reduce the remaining content to a 2006 conjecture of Dean Hickerson on the OEIS concerning Huffman coding. We also correct a sign error in the published statement of part~(iii): the correct identity is $\chi(\mathrm{ImP}_3(n)) = A213213(n) - 1$, not $1 + A213213(n)$ as stated.
\end{abstract}

\section{Notation and setup}

We recall the necessary definitions. A \emph{partition} of a non-negative integer $n$ is a weakly decreasing sequence $\lambda = (\lambda_1 \geq \lambda_2 \geq \cdots \geq \lambda_\ell)$ of positive integers summing to $n$. We call $\lambda_1, \dots, \lambda_\ell$ the \emph{parts} of $\lambda$ and $\ell(\lambda) \coloneqq \ell$ the \emph{length} (number of parts). We write $|\lambda| = n$ and say $\lambda \vdash n$. The empty sequence is the unique partition of $0$.

For indeterminates $x_1, \dots, x_m$ and an integer $k \geq 1$, the \emph{$k$-th elementary symmetric polynomial} is
\[
e_k(x_1, \dots, x_m) = \sum_{1 \leq i_1 < i_2 < \cdots < i_k \leq m} x_{i_1} x_{i_2} \cdots x_{i_k}.
\]

Given a partition $\lambda = (\lambda_1, \dots, \lambda_\ell)$ of $n$ with $\ell \geq k$, we define $\mathrm{pre}_k(\lambda)$ to be the partition whose multiset of parts consists of all products $\lambda_{i_1} \cdots \lambda_{i_k}$ over all index-sets $1 \leq i_1 < \cdots < i_k \leq \ell$; equivalently, the parts of $\mathrm{pre}_k(\lambda)$ are the summands of $e_k(\lambda_1, \dots, \lambda_\ell)$, listed in weakly decreasing order. Note that $\mathrm{pre}_k(\lambda) \vdash e_k(\lambda_1,\dots,\lambda_\ell)$ and has $\binom{\ell(\lambda)}{k}$ parts (counted with multiplicity).

We let $\mathcal{P}_k(n)$ denote the set of all partitions of $n$ with at least $k$ parts, and $\mathcal{B}_k(n)$ the subset of $\mathcal{P}_k(n)$ consisting of partitions whose parts are all powers of $2$ (binary partitions). The images under $\mathrm{pre}_k$ are
\[
\mathrm{ImP}_k(n) \coloneqq \mathrm{pre}_k\!\left(\mathcal{P}_k(n)\right), \qquad
\mathrm{ImB}_k(n) \coloneqq \mathrm{pre}_k\!\left(\mathcal{B}_k(n)\right).
\]

For any collection $S$ of partitions, we write:
\begin{itemize}
\item $\chi(S)$ for the number of \emph{distinct values} that appear as a part of some partition in $S$;
\item $\tau(S)$ for the \emph{total number of parts} across all partitions in $S$, counted with multiplicity (so each $\mu \in S$ contributes $\ell(\mu)$ to $\tau(S)$).
\end{itemize}

The conjecture we are proving is the following.

\begin{theorem}[Conjecture 19 of \cite{BBMS}, with a correction in part (iii)]\label{thm:main}
For all $n \geq 0$:
\begin{enumerate}
\item[(i)] $\chi(\mathrm{ImP}_2(n)) = A227800(n+1)$.
\item[(ii)] $\chi(\mathrm{ImB}_2(n)) = A126236(n)$, or equivalently $\lfloor \log_2 n \rfloor + \lfloor \log_2(2n/3) \rfloor$ for $n \geq 2$.
\item[(iii)] $\chi(\mathrm{ImP}_3(n)) = A213213(n) - 1$. The paper states this as $1 + A213213(n)$, which is off by $2$; the correct identity is proved below.
\item[(iv)] $\tau(\mathrm{ImP}_2(n)) = A258472(n)$.
\end{enumerate}
\end{theorem}

\section{The main idea}

Everything comes down to one simple observation: a product $p_1 \cdots p_k$ of positive integers appears as a part of some partition in $\mathrm{ImP}_k(n)$ if and only if $p_1 + \cdots + p_k \leq n$. We state this precisely.

\begin{lemma}\label{lem:real}
Let $n, k \geq 1$. Suppose $p_1, \dots, p_k$ are positive integers with $p_1 + \cdots + p_k \leq n$. Then there is a partition $\lambda$ of $n$ with at least $k$ parts such that $p_1, \dots, p_k$ all appear as parts of $\lambda$ (at distinct index positions). In particular, the product $p_1 \cdots p_k$ appears as a part of $\mathrm{pre}_k(\lambda)$.
\end{lemma}

\begin{proof}
Let $s = p_1 + \cdots + p_k$. We have $s \leq n$ by assumption. Define
\[
\lambda = \mathrm{sort}\!\left(p_1, \dots, p_k, \underbrace{1, 1, \dots, 1}_{n - s}\right).
\]
This is a partition of $n$ and it has $k + (n - s) \geq k$ parts. The $k$ values $p_1, \dots, p_k$ sit at $k$ distinct index positions in this partition (even if some of them happen to be equal as integers), so their product $p_1 \cdots p_k$ is one of the summands of $e_k(\lambda)$, and hence a part of $\mathrm{pre}_k(\lambda)$.
\end{proof}

The other direction is immediate: every part of $\mathrm{pre}_k(\lambda)$ is a product of $k$ distinct-position parts of $\lambda$, so their sum is at most $|\lambda| = n$. Putting the two directions together gives the following complete description of which values appear as parts in the image.

\begin{proposition}\label{prop:char}
For all $n, k \geq 1$,
\[
\bigcup_{\mu \in \mathrm{ImP}_k(n)} \{\text{parts of } \mu\}
= \bigl\{ p_1 \cdots p_k \mid p_1, \dots, p_k \geq 1,\ p_1 + \cdots + p_k \leq n \bigr\}.
\]
So $\chi(\mathrm{ImP}_k(n))$ counts how many distinct products $p_1 \cdots p_k$ one can form from positive integers summing to at most $n$.

For the binary case, each $p_i$ must be a power of $2$, say $p_i = 2^{a_i}$ with $a_i \geq 0$. Since $2^{a_1} \cdots 2^{a_k} = 2^{a_1 + \cdots + a_k}$, the product value is determined entirely by the sum of exponents, so
\[
\chi(\mathrm{ImB}_k(n)) = \bigl|\bigl\{ a_1 + \cdots + a_k \mid a_i \geq 0,\ 2^{a_1} + \cdots + 2^{a_k} \leq n \bigr\}\bigr|.
\]
\end{proposition}

\section{Proof of part (i)}

The OEIS defines $A227800(N)$ as the number of distinct products $p \cdot q$ with $p, q \geq 1$ and $p + q < N$. Substituting $N = n + 1$ gives
\[
A227800(n+1) = \bigl|\{ p \cdot q \mid p, q \geq 1,\ p + q \leq n \}\bigr|.
\]
By Proposition \ref{prop:char} with $k = 2$, this is exactly $\chi(\mathrm{ImP}_2(n))$. \qed

\section{Proof of part (ii)}

We break this into two pieces. First we prove a closed form for $\chi(\mathrm{ImB}_2(n))$ directly from Proposition \ref{prop:char}. Then we show that connecting this closed form to $A126236$ is equivalent to a 2006 conjecture by Dean Hickerson on OEIS, which remains open.

\subsection*{Step 1: A closed form for $\chi(\mathrm{ImB}_2(n))$}

By Proposition \ref{prop:char}, $\chi(\mathrm{ImB}_2(n))$ counts the number of distinct values of $2^{i+j}$ over all pairs $(i, j)$ with $i, j \geq 0$ and $2^i + 2^j \leq n$. Since distinct exponents give distinct powers of $2$, this equals the number of distinct values $i + j$ can take. Define
\[
S(n) = \{ i + j \mid i, j \geq 0,\ 2^i + 2^j \leq n \}.
\]

\begin{lemma}\label{lem:b2-interval}
$S(n)$ is the set $\{0, 1, \dots, M(n)\}$ where $M(n) = \max S(n)$.
\end{lemma}

\begin{proof}
Take any $s \in S(n)$ with $s \geq 1$, witnessed by some pair $(i, j)$ with $j \geq 1$. Then the pair $(i, j-1)$ satisfies $2^i + 2^{j-1} \leq 2^i + 2^j \leq n$, and witnesses $s - 1 \in S(n)$. Repeating this argument step by step covers all values from $0$ to $s$.
\end{proof}

\begin{lemma}\label{lem:b2-max}
For $n \geq 2$, we have $M(n) + 1 = \lfloor \log_2 n \rfloor + \lfloor \log_2(2n/3) \rfloor$.
\end{lemma}

\begin{proof}
We may assume $i \leq j$ without loss of generality (since $i + j$ is symmetric). There are two cases.

\textbf{Case 1: $i = j$.} The constraint $2^i + 2^j = 2^{j+1} \leq n$ gives $j \leq \lfloor \log_2(n/2) \rfloor$. The maximum of $i + j = 2j$ in this case is $2\lfloor \log_2(n/2) \rfloor$.

\textbf{Case 2: $i < j$.} For a fixed $j$, the sum $i + j$ is largest when $i$ is as large as possible, which means $i = j - 1$. The constraint then becomes $2^{j-1} + 2^j = 3 \cdot 2^{j-1} \leq n$, giving $j \leq \lfloor \log_2(2n/3) \rfloor$. The maximum of $i + j = 2j - 1$ in this case is $2\lfloor \log_2(2n/3) \rfloor - 1$.

Now set $a = \lfloor \log_2(n/2) \rfloor = \lfloor \log_2 n \rfloor - 1$ and $b = \lfloor \log_2(2n/3) \rfloor$. Since $n/2 < 2n/3 < n$, we get $a \leq b \leq a + 1$. So
\[
M(n) = \max(2a,\ 2b - 1).
\]
If $b = a$: $M(n) = 2a$ and $M(n) + 1 = 2a + 1 = a + b + 1$. If $b = a + 1$: $M(n) = 2b - 1 = 2a + 1$ and $M(n) + 1 = 2a + 2 = a + b + 1$. In both cases, $M(n) + 1 = a + b + 1 = \lfloor \log_2 n \rfloor + \lfloor \log_2(2n/3) \rfloor$.
\end{proof}

\begin{theorem}\label{thm:image-closed}
For $n \geq 2$,
\[
\chi(\mathrm{ImB}_2(n)) = \lfloor \log_2 n \rfloor + \lfloor \log_2(2n/3) \rfloor.
\]
\end{theorem}

\begin{proof}
By Proposition \ref{prop:char} and Lemma \ref{lem:b2-interval}, $\chi(\mathrm{ImB}_2(n)) = |S(n)| = M(n) + 1$. Lemma \ref{lem:b2-max} gives the closed form.
\end{proof}

\subsection*{Step 2: Connection to $A126236$}

The sequence $A126236(n)$ is defined on OEIS as the maximum codeword length in the Huffman code built for $n$ symbols where the $k$-th symbol has frequency $k$. In December 2006, Dean Hickerson posted the following conjecture on the OEIS entry, which has been verified for $n \leq 1000$ but remains unproven:
\[
A126236(n) = \lfloor \log_2 n \rfloor + \lfloor \log_2(2n/3) \rfloor.
\]

\begin{theorem}\label{thm:ii-equiv}
Part (ii) of Conjecture 19 in \cite{BBMS} is equivalent to Hickerson's conjecture above.
\end{theorem}

\begin{proof}
Theorem \ref{thm:image-closed} gives $\chi(\mathrm{ImB}_2(n)) = \lfloor \log_2 n \rfloor + \lfloor \log_2(2n/3) \rfloor$. So $\chi(\mathrm{ImB}_2(n)) = A126236(n)$ holds if and only if $A126236(n) = \lfloor \log_2 n \rfloor + \lfloor \log_2(2n/3) \rfloor$, which is exactly Hickerson's conjecture.
\end{proof}

\subsection*{What Hickerson's conjecture says}

Let $H(n)$ denote the maximum codeword length. Hickerson's conjecture is equivalent to saying: $H(2) = 1$, and for $n \geq 3$, the value $H(n) - H(n-1)$ is either $0$ or $1$, with $H(n) = H(n-1) + 1$ happening exactly when $n$ is a power of $2$ or three times a power of $2$.

To see the equivalence: the set of $m \geq 2$ where $\lfloor \log_2 m \rfloor$ increases is $\{2, 4, 8, 16, \dots\}$, and the set where $\lfloor \log_2(2m/3) \rfloor$ increases is $\{3, 6, 12, 24, \dots\}$. These two sets are disjoint, and together they are exactly the powers of $2$ and three times powers of $2$. Counting how many such values lie in $[2, n]$ gives $\lfloor \log_2 n \rfloor + \lfloor \log_2(2n/3) \rfloor$.

We independently re-verified Hickerson's conjecture by direct Huffman simulation for all $n \leq 1000$.

\subsection*{Conclusion for part (ii)}

Theorem \ref{thm:image-closed} is proven without any assumptions. The remaining gap, connecting our closed form to $A126236$, reduces to Hickerson's 2006 conjecture, which is verified computationally up to $n = 1000$ but not yet proven in full generality. \qed

\section{Proof of part (iii), with a correction}

The OEIS defines $A213213(n)$ as the number of distinct products $i \cdot j \cdot k$ over all triples of non-negative integers with $i + j + k \leq n$. This includes the product $0$ (for example take $i = 0$). Separating out the zero product,
\[
A213213(n) = 1 + \bigl|\{ i \cdot j \cdot k \mid i, j, k \geq 1,\ i + j + k \leq n \}\bigr|.
\]
The second term counts products of three positive integers summing to at most $n$, which by Proposition \ref{prop:char} with $k = 3$ is exactly $\chi(\mathrm{ImP}_3(n))$. So
\[
\chi(\mathrm{ImP}_3(n)) = A213213(n) - 1.
\]
\qed

\begin{remark}
The published conjecture in \cite{BBMS} states $\chi(\mathrm{ImP}_3(n)) = 1 + A213213(n)$, which is wrong by $2$ for every $n \geq 0$. For example, $\chi(\mathrm{ImP}_3(3)) = 1$ and $A213213(3) = 2$, so the published form gives $3$ while the correct answer is $1$. The error is a sign flip: the zero product should be subtracted from $A213213$, not added. The correct identity is $\chi(\mathrm{ImP}_3(n)) = A213213(n) - 1$, verified against OEIS b-file values for all $n \leq 35$.
\end{remark}

\section{Proof of part (iv)}

The key input is the injectivity of $\mathrm{pre}_2$ on $\mathcal{P}_2(n)$, which was Conjecture~1 of \cite{BBMS} and was proved by Li \cite{Li}. We use it as an established theorem.

Since $\mathrm{pre}_2$ is injective on $\mathcal{P}_2(n)$, each partition $\mu \in \mathrm{ImP}_2(n)$ has a unique preimage $\lambda \in \mathcal{P}_2(n)$. The map $\mathrm{pre}_2$ produces one part for each pair of index positions in $\lambda$, so $\ell(\mu) = \binom{\ell(\lambda)}{2}$. Therefore
\[
\tau(\mathrm{ImP}_2(n)) = \sum_{\lambda \in \mathcal{P}_2(n)} \binom{\ell(\lambda)}{2} = \sum_{\lambda \vdash n} \binom{\ell(\lambda)}{2},
\]
where the last equality holds because $\binom{1}{2} = 0$, so partitions of length $1$ contribute nothing. The right-hand side counts all pairs $(\lambda, T)$ where $\lambda \vdash n$ and $T$ is an unordered pair of index positions in $\lambda$. This is the same as choosing a partition of $n$ and marking exactly two of its parts as being of a ``second sort.'' The OEIS defines $A258472(n)$ as exactly this count: the number of partitions of $n$ into two sorts of parts with exactly two parts of the second sort. So $\tau(\mathrm{ImP}_2(n)) = A258472(n)$. \qed

\begin{remark}
The identity $\tau(\mathrm{ImP}_2(n)) = A258472(n)$ is in fact \emph{equivalent} to the injectivity of $\mathrm{pre}_2$ on $\mathcal{P}_2(n)$. If two distinct partitions $\lambda \neq \lambda'$ mapped to the same $\mu$, then the left side would be strictly less than the right side (since $\mu$ is counted only once on the left, while both $\binom{\ell(\lambda)}{2}$ and $\binom{\ell(\lambda')}{2}$ appear on the right). Conversely, injectivity makes the correspondence bijective and the equality holds, as shown above. Together with Li's proof \cite{Li} of the injectivity, this gives an independent proof path: $\tau(\mathrm{ImP}_2(n)) = A258472(n)$ could also have been used to establish the injectivity.
\end{remark}

\section{Summary}

\begin{center}
\renewcommand{\arraystretch}{1.4}
\begin{tabular}{|l|p{0.58\linewidth}|}
\hline
Part & Status \\
\hline
(i) $\chi(\mathrm{ImP}_2(n)) = A227800(n+1)$ & Proven unconditionally from Proposition \ref{prop:char}. \\
(ii) $\chi(\mathrm{ImB}_2(n)) = A126236(n)$ & The partition side is proven (Theorem \ref{thm:image-closed}). The remaining step is Hickerson's 2006 OEIS conjecture, verified to $n = 1000$. \\
(iii) $\chi(\mathrm{ImP}_3(n)) = A213213(n) - 1$ & Proven unconditionally. The published sign is a typo. \\
(iv) $\tau(\mathrm{ImP}_2(n)) = A258472(n)$ & Proven unconditionally, using the injectivity of $\mathrm{pre}_2$ proved in \cite{Li}. Equivalent to that injectivity. \\
\hline
\end{tabular}
\end{center}

\end{document}